\documentclass[11pt,letterpaper,twoside,reqno]{amsart}

\usepackage{amsmath}
\usepackage{amssymb}
\usepackage{bm}
\usepackage{mathrsfs}
\usepackage[dvips]{graphicx}
\usepackage{stmaryrd}
\usepackage{wasysym}

\usepackage{color}

\usepackage{url}

\usepackage{supertabular}

\usepackage{alg}


\newcommand{\algfor}[1]{\textbf{for} #1\\\algbegin}

\newtheorem{thm}{Theorem}
\newtheorem{cor}[thm]{Corollary}

\newtheorem{prop}[thm]{Proposition}

\newtheorem{rem}[thm]{Remark}

\numberwithin{equation}{section}
\numberwithin{thm}{section}

\DeclareMathAlphabet{\mathsfsl}{OT1}{cmss}{m}{sl}

\newcommand{\term}{\emph}

\newcommand{\cnst}[1]{\mathrm{#1}}

	{\makebox{\phantom{}} \\ %
		\textsc{Input:}\begin{itemize}}%
	{\end{itemize}}

	{\textsc{Output:}\begin{itemize}}%
	{\end{itemize}}

	{\textsc{Procedure:}\begin{enumerate}}%
	{\end{enumerate}}

\renewcommand{\phi}{\varphi}

\newcommand{\eps}{\varepsilon}

\newcommand{\econst}{\mathrm{e}}

\newcommand{\onevct}{\mathbf{e}}

\newcommand{\Id}{\mathbf{I}}

\newcommand{\Rspace}[1]{\mathbb{R}^{#1}}
\newcommand{\Cspace}[1]{\mathbb{C}^{#1}}

\newcommand{\abs}[1]{\left\vert {#1} \right\vert}
\newcommand{\abssq}[1]{{\abs{#1}}^2}

\newcommand{\argmin}{\operatorname*{arg\; min}}

\newcommand{\psdge}{\succcurlyeq}
\newcommand{\psdle}{\preccurlyeq}

\newcommand{\Expect}{\operatorname{\mathbb{E}}}

\newcommand{\vct}[1]{\bm{#1}}
\newcommand{\mtx}[1]{\bm{#1}}

\newcommand{\adj}{*}
\newcommand{\psinv}{\dagger}

\newcommand{\rank}{\operatorname{rank}}
\newcommand{\strank}{\operatorname{st.rank}}

\newcommand{\diag}{\operatorname{diag}}
\newcommand{\trace}{\operatorname{trace}}

\newcommand{\ip}[2]{\left\langle {#1},\ {#2} \right\rangle}
\newcommand{\absip}[2]{\abs{\ip{#1}{#2}}}

\newcommand{\norm}[1]{\left\Vert {#1} \right\Vert}
\newcommand{\normsq}[1]{\norm{#1}^2}

\newcommand{\enorm}[1]{\norm{#1}_2}
\newcommand{\enormsq}[1]{\enorm{#1}^2}

\newcommand{\fnorm}[1]{\norm{#1}_{\mathrm{F}}}
\newcommand{\fnormsq}[1]{\fnorm{#1}^2}

\newcommand{\pnorm}[2]{\norm{#2}_{#1}}
\newcommand{\infnorm}[1]{\norm{#1}_{\infty}}

\newcommand{\triplenorm}[1]{\left\vert\!\left\vert\!\left\vert {#1} \right\vert\!\right\vert\!\right\vert}

\newcommand{\conv}{\operatorname{conv}}

\newcommand{\subjto}{\quad\text{subject to}\quad}

\newcommand{\bigO}{\mathrm{O}}

\begin{document}

\title[Column Subset Selection]
{Column Subset Selection, Matrix Factorization,\\ and Eigenvalue Optimization}

\author{J.~A.~Tropp}

\thanks{JAT is with Applied and Computational Mathematics, MC 217-50, California Inst.~Technology, Pasadena, CA 91125-5000.
E-mail: \url{jtropp@acm.caltech.edu}.
Supported in part by ONR award no. N000140810883.}

\date{26 June 2008}

\begin{abstract}
Given a fixed matrix, the problem of column subset selection requests a column submatrix that has favorable spectral properties.  Most research from the algorithms and numerical linear algebra communities focuses on a variant called rank-revealing {\sf QR}, which seeks a well-conditioned collection of columns that spans the (numerical) range of the matrix.  The functional analysis literature contains another strand of work on column selection whose algorithmic implications have not been explored.  In particular, a celebrated result of Bourgain and Tzafriri demonstrates that each matrix with normalized columns contains a large column submatrix that is exceptionally well conditioned.  Unfortunately, standard proofs of this result cannot be regarded as algorithmic.

This paper presents a randomized, polynomial-time algorithm that produces the submatrix promised by Bourgain and Tzafriri.  The method involves random sampling of columns, followed by a matrix factorization that exposes the well-conditioned subset of columns.  This factorization, which is due to Grothendieck, is regarded as a central tool in modern functional analysis.  The primary novelty in this work is an algorithm, based on eigenvalue minimization, for constructing the Grothendieck factorization.  These ideas also result in a novel approximation algorithm for the $(\infty, 1)$ norm of a matrix, which is generally {\sf NP}-hard to compute exactly.  As an added bonus, this work reveals a surprising connection between matrix factorization and the famous {\sc maxcut} semidefinite program.
\end{abstract}



\maketitle


\newpage

\section{Introduction}

\term{Column subset selection} refers to the challenge of extracting from a matrix a column submatrix that has some distinguished property.  These properties commonly involve conditions on the spectrum of the submatrix.  The most familiar example is probably rank-revealing {\sf QR}, which seeks a well-conditioned collection of columns that spans the (numerical) range of the matrix \cite{GE96:Efficient-Algorithms}.

The literature on geometric functional analysis contains several fundamental theorems on column subset selection that have not been discussed by the algorithms community or the numerical linear algebra community.  These results are phrased in terms of the \term{stable rank} of a matrix:
$$
\strank(\mtx{A}) = \frac{\fnormsq{\mtx{A}}}{\normsq{\mtx{A}}}
$$
where $\fnorm{\cdot}$ is the Frobenius norm and $\norm{\cdot}$ is the spectral norm.  The stable rank can be viewed as an analytic surrogate for the algebraic rank.  Indeed, express the two norms in terms of singular values to obtain the relation
$$
\strank(\mtx{A}) \leq \rank(\mtx{A}).
$$
In this bound, equality occurs (for example) when the columns of $\mtx{A}$ are identical or when the columns of $\mtx{A}$ are orthonormal.  As we will see, the stable rank is tightly connected with the number of (strongly) linearly independent columns we can extract from a matrix.

Before we continue, let us instate some notation.  We say that a matrix is \term{standardized} when its columns have unit $\ell_2$ norm.  The $j$th column of a matrix $\mtx{A}$ is denoted by $\vct{a}_j$.  For a subset $\tau$ of column indices, we write $\mtx{A}_{\tau}$ for the column submatrix indexed by $\tau$.  Likewise, given a square matrix $\mtx{H}$, the notation $\mtx{H}_{\tau \times \tau}$ refers to the principal submatrix whose rows and columns are listed in $\tau$.  The pseudoinverse $\mtx{D}^\psinv$ of a diagonal matrix $\mtx{D}$ is formed by reciprocating the nonzero entries.
As usual, we write $\pnorm{p}{\cdot}$ for the $\ell_p$ vector norm.  The \term{condition number} of a matrix is the quantity
$$
\kappa(\mtx{A}) =
\max\left\{ \frac{\enorm{\mtx{A}\vct{x}}}{\enorm{\mtx{A}\vct{y}}} :
	\enorm{\vct{x}} = \enorm{\vct{y}} = 1 \right\}.
$$
Finally, upright letters $(\cnst{c}, \cnst{C}, \cnst{K}, \dots)$ refer to positive, universal constants that may change from appearance to appearance.

The first theorem, due to Kashin and Tzafriri, shows that each matrix with standardized columns contains a large column submatrix that has small spectral norm~\cite[Thm.~2.5]{Ver01:Johns-Decompositions}.


\begin{thm}[Kashin--Tzafriri] \label{thm:kt}
Suppose $\mtx{A}$ is standardized.  Then there is a set $\tau$ of column indices for which
$$
\abs{\tau} \geq \strank(\mtx{A})
\quad\text{and}\quad
\norm{\mtx{A}_{\tau}} \leq \cnst{C}.
$$
\end{thm}


In fact, much more is true.  Combining Theorem~\ref{thm:kt} with the celebrated restricted invertibility result of Bourgain and Tzafriri~\cite[Thm.~1.2]{BT87:Invertibility-Large}, we find that every standardized matrix contains a large column submatrix whose \emph{condition number} is small.  



\begin{thm}[Bougain--Tzafriri] \label{thm:bt}
Suppose $\mtx{A}$ is standardized.  Then there is a set $\tau$ of column indices for which
$$
\abs{\tau} \geq \cnst{c} \cdot \strank(\mtx{A})
\quad\text{and}\quad
\kappa(\mtx{A}_{\tau}) \leq \sqrt{3}.
$$
\end{thm}

Theorem~\ref{thm:bt} yields the best general result~\cite[Thm.~1.1]{BT91:Problem-Kadison-Singer} on the Kadison--Singer conjecture, the major open question in operator theory.  To display its strength, let us consider two extreme examples.

\begin{enumerate} 
\item	When $\mtx{A}$ has identical columns, every collection of two or more columns is singular.  Theorem~\ref{thm:bt} guarantees a well-conditioned submatrix $\mtx{A}_\tau$ with $\abs{\tau} = 1$, which is optimal.

\item	When $\mtx{A}$ has $n$ orthonormal columns, the full matrix is perfectly conditioned.  Theorem~\ref{thm:bt} guarantees a well-conditioned submatrix $\mtx{A}_\tau$ with $\abs{\tau} \geq \cnst{c}n$, which lies within a constant factor of optimal.
\end{enumerate}
The stable rank allows Theorem~\ref{thm:bt} to interpolate between the two extremes.  Subsequent research established that the stable rank is intrinsic to the problem of finding well-conditioned submatrices.  We postpone a more detailed discussion of this point until Section~\ref{sec:discussion}.


\subsection{Contributions}

Although Theorems~\ref{thm:kt} and~\ref{thm:bt} would be very useful in computational applications, we cannot regard current proofs as constructive.  The goal of this paper is to establish the following novel, algorithmic claim.

\begin{thm}
There are randomized, polynomial-time algorithms for producing the sets guaranteed by Theorem~\ref{thm:kt} and by Theorem~\ref{thm:bt}.
\end{thm}

This result is significant because no known algorithm for column subset selection is guaranteed to produce a submatrix whose condition number has constant order.  See~\cite{BDM08:Selecting-Exactly} for a recent overview of that literature.
The present work has other ramifications with independent interest.
\begin{itemize}
\item	We develop algorithms for computing the matrix factorizations of Pietsch and Grothendieck, which are regarded as basic instruments in modern functional analysis~\cite{Pis86:Factorization-Linear}.

\item	The methods for computing these factorizations lead to new approximation algorithms for two {\sf NP}-hard matrix norms.  (See Remarks~\ref{rem:inf2-norm} and~\ref{rem:inf1-norm}.)

\item 	We identify an intriguing connection between Pietsch factorization and the {\sc maxcut} semidefinite program~\cite{GW95:Improved-Approximation}.
\end{itemize}

\subsection{Overview}

We focus on the algorithmic version of the Kashin--Tzafriri theorem because it highlights all the essential concepts while minimizing irrelevant details.  Section~\ref{sec:kt-pf} outlines a proof of this result, emphasizing where new algorithmic machinery is required.  The missing link turns out to be a computational method for producing a certain matrix factorization.  Section~\ref{sec:pietsch} reformulates the factorization problem as an eigenvalue minimization, which can be completed with standard techniques.  In Section~\ref{sec:kt-alg}, we exhibit a randomized algorithm that delivers the submatrix promised by Kashin--Tzafriri.  In Section~\ref{sec:bt}, we traverse a similar route to develop an algorithmic version of Bourgain--Tzafriri.  Section~\ref{sec:discussion} provides more details about the stable rank and describes directions for future work.  Appendix~\ref{app:rdm-norm} contains some key estimates on the norms of random submatrices, and Appendix~\ref{app:emd} outlines a simple computational procedure for solving the eigenvalue optimization problems that arise in our work.

\section{The Kashin--Tzafriri Theorem}
	\label{sec:kt-pf}

The proof of the Kashin--Tzafriri theorem proceeds in two steps.  First, we select a random set of columns with appropriate cardinality.  Second, we use a matrix factorization to identify and remove redundant columns that inflate the spectral norm.  The proof gives strong hints about how a computational procedure might work, even though it is not constructive.

\subsection{Intuitions}

We would like to think that a random submatrix inherits its share of the norm of the entire matrix.  In other words, if we were to select a tenth of the columns, we might hope to reduce the norm by a factor of ten.  Unfortunately, this intuition is meretricious.

Indeed, random selection does not necessarily reduce the spectral norm at all.  The essential reason emerges when we consider the ``double identity,'' the $m \times 2m$ matrix $\mtx{A} = \begin{bmatrix} \Id \ |\ \Id \end{bmatrix}$.  Suppose we draw $s$ random columns from $\mtx{A}$ without replacement.  The probability that all $s$ columns are distinct is
$$
\frac{2m - 2}{2m - 1} \times \frac{2m - 4}{2m - 2} \times \cdots \times
	\frac{2m - 2(s-1)}{2m - (s-1)}
	\leq \prod\nolimits_{j=0}^{s-1} \left( 1 - \frac{j}{2m} \right)
	\approx \exp\left\{ - \sum\nolimits_{j=0}^{s-1} \frac{j}{2m} \right\}
	\approx \econst^{-s^2/4m}.
$$
Therefore, when $s = \Omega(\sqrt{m})$, sampling almost always produces a submatrix with at least one duplicated column.  A duplicated column means that the norm of the submatrix is $\sqrt{2}$, which equals the norm of the full matrix, so no reduction takes place.

Nevertheless, a randomly chosen set of columns from a standardized matrix typically \emph{contains} a large set of columns that has small norm.  We will see that the desired subset is exposed by factoring the random submatrix.  This factorization, which was invented by Pietsch, is regarded as a basic instrument in modern functional analysis.

\subsection{The $(\infty,2)$ operator norm}

Although sampling does not necessarily reduce the spectral norm, it often reduces other matrix norms.  Define the natural norm on linear operators from $\ell_\infty$ to $\ell_2$ via the expression
$$
\pnorm{\infty\to2}{\mtx{B}}
	= \max\{ \enorm{\mtx{B}\vct{x}} : \infnorm{\mtx{x}} = 1 \}.
$$
An immediate consequence is that
$\pnorm{\infty\to2}{\mtx{B}} \leq \sqrt{s}\norm{\mtx{B}}$
for each matrix $\mtx{B}$ with $s$ columns.  Equality can obtain in this bound.  

The exact calculation of the $(\infty, 2)$ operator norm is computationally difficult.  Results of Rohn~\cite{Roh00:Computing-Norm} imply that there is a class of positive semidefinite matrices for which it is {\sf NP}-hard to estimate $\pnorm{\infty\to2}{\cdot}$ within an absolute tolerance.  Nevertheless, we will see that the norm can be approximated in polynomial time up to a small relative error.  (See Remark~\ref{rem:inf2-norm}.)



As we have intimated, the $(\infty, 2)$ norm can often be reduced by random selection.  The following theorem requires some heavy lifting, which we delegate to Appendix~\ref{app:inf2}.

\begin{thm} \label{thm:rdm-inf2}
Suppose $\mtx{A}$ is a standardized matrix with $n$ columns.  Choose
$$
s \leq \lceil 2 \strank(\mtx{A}) \rceil,
$$
and draw a uniformly random subset $\sigma$ with cardinality $s$ from $\{1, 2, \dots, n\}$.  Then
$$
\Expect \pnorm{\infty\to 2}{\mtx{A}_{\sigma}} \leq
	7 \sqrt{s}.
$$
In particular, $\pnorm{\infty\to 2}{\mtx{A}_{\sigma}} \leq 8 \sqrt{s}$ with probability at least $1/8$.
\end{thm}



\subsection{Pietsch factorization}

We cannot exploit the bound in Theorem~\ref{thm:rdm-inf2} unless we have a way to connect the $(\infty, 2)$ norm with the spectral norm.  To that end, let us recall one of the landmark theorems of functional analysis.

\begin{thm}[Pietsch Factorization] \label{thm:pietsch}
Each matrix $\mtx{B}$ can be factored as $\mtx{B} = \mtx{TD}$ where
\begin{itemize}
\item	$\mtx{D}$ is a nonnegative, diagonal matrix with $\trace(\mtx{D}^2) = 1$, and
\item	$\pnorm{\infty\to2}{\mtx{B}} \leq \norm{\mtx{T}} \leq \cnst{K_P} \pnorm{\infty\to 2}{\mtx{B}}$.
\end{itemize}
\end{thm}

This result follows from the little Grothendieck theorem~\cite[Sec.~5b]{Pis86:Factorization-Linear} and the Pietsch factorization theorem~\cite[Cor.~1.8]{Pis86:Factorization-Linear}.  The standard proof produces the factorization using an abstract separation argument that offers no algorithmic insight.  The value of the constant is available.
\begin{itemize}
\item	When the scalar field is real, we have $\cnst{K_P}(\Rspace{}) = \sqrt{\pi/2} \approx 1.25$.
\item	When the scalar field is complex, we have $\cnst{K_P}(\Cspace{}) = \sqrt{4/\pi} \approx 1.13$.
\end{itemize}

A major application of Pietsch factorization is to identify a submatrix with controlled spectral norm.  The following proposition describes the procedure.

\begin{prop} \label{prop:pietsch-cols}
Suppose $\mtx{B}$ is a matrix with $s$ columns.  Then there is a set $\tau$ of column indices for which
$$
\abs{\tau} \geq \frac{s}{2}
\quad\text{and}\quad
\norm{ \mtx{B}_{\tau} }
	\leq \cnst{K_P} \sqrt{\frac{2}{s}} \pnorm{\infty\to2}{\mtx{B}}.
$$
\end{prop}

\begin{proof}
Consider a Pietsch factorization $\mtx{B} = \mtx{TD}$, and define
$$
\tau = \{ j : d_{jj}^2 \leq 2/s \}.
$$
Since $\sum d_{jj}^2 = 1$, Markov's inequality implies that $\abs{\tau} \geq s/2$.  We may calculate that
$$
\norm{ \mtx{B}_{\tau} }
	= \norm{ \mtx{TD}_{\tau} }
	\leq \norm{ \mtx{T} } \cdot \norm{ \mtx{D}_{\tau} }
	\leq \cnst{K_P} \pnorm{\infty\to2}{\mtx{B}} \cdot \sqrt{2/s}.
$$
This completes the proof.
\end{proof}

\subsection{Proof of Kashin--Tzafriri}
	\label{sec:kt-pf-digest}

With these results at hand, we easily complete the proof of the Kashin--Tzafriri theorem.  Suppose $\mtx{A}$ is a standardized matrix with $n$ columns.  Assume that $\strank(\mtx{A}) \leq n/2$.  Otherwise, the spectral norm $\norm{\mtx{A}} \leq \sqrt{2}$, so we may select $\tau = \{1, 2, \dots, n\}$.

According to Theorem~\ref{thm:rdm-inf2}, there is a subset $\sigma$ of column indices for which
$$
\abs{\sigma} \geq 2 \strank(\mtx{A})
\quad\text{and}\quad
\pnorm{\infty\to2}{ \mtx{A}_{\sigma} } \leq 8 \sqrt{\abs{\sigma}}.
$$
Apply Proposition~\ref{prop:pietsch-cols} to the matrix $\mtx{B} = \mtx{A}_{\sigma}$ to obtain a subset $\tau$ inside $\sigma$ for which
$$
\abs{\tau} \geq \frac{\abs{\sigma}}{2}
\quad\text{and}\quad
\norm{ \mtx{B}_{\tau} }
	\leq \cnst{K_P} \sqrt{\frac{2}{\abs{\sigma}}}\pnorm{\infty\to2}{\mtx{B}}.
$$
Since $\mtx{B}_{\tau} = \mtx{A}_{\tau}$ and $\cnst{K_P} \leq \sqrt{\pi/2}$, these bounds reveal the advertised conclusion:
$$
\abs{\tau} \geq \strank(\mtx{A})
\quad\text{and}\quad
\norm{ \mtx{A}_{\tau} }
	< 15.
$$


At this point, we take a step back and notice that this proof is nearly algorithmic.  It is straightforward to perform the random selection described in Theorem~\ref{thm:rdm-inf2}.  Provided that we know a Pietsch factorization of the matrix $\mtx{B}$, we can easily carry out the column selection  of Proposition~\ref{prop:pietsch-cols}.  Therefore, we need only develop an algorithm for computing the Pietsch factorization to reach an effective version of the Kashin--Tzafriri theorem.

\section{Pietsch Factorization via Convex Optimization}
	\label{sec:pietsch}

The main novelty is to demonstrate that we can produce a Pietsch factorization by solving a convex programming problem.  Remarkably, the resulting optimization is the dual of the famous \textsc{maxcut} semidefinite program~\cite{GW95:Improved-Approximation}, for which many polynomial-time algorithms are available.

\subsection{Pietsch and eigenvalues}
	\label{sec:pietsch-eigs}

The next theorem, which serves as the basis for our computational method, demonstrates that Pietsch factorizations have an intimate relationship with the eigenvalues of a related matrix.  In the sequel, we reserve the letter $\mtx{D}$ for a nonnegative, diagonal matrix with $\trace(\mtx{D}^2) = 1$, and we write $\lambda_{\max}$ for the algebraically maximal eigenvalue of a Hermitian matrix.

\begin{thm} \label{thm:pietsch-eig}
The factorization $\mtx{B} = \mtx{TD}$ satisfies $\norm{\mtx{T}} \leq \alpha$ if and only if $\mtx{D}$ satisfies
$$
\lambda_{\max}(\mtx{B}^\adj \mtx{B} - \alpha^2 \mtx{D}^2) \leq 0.
$$
In particular, if no $\mtx{D}$ verifies this bound, then no factorization $\mtx{B} = \mtx{TD}$ admits $\norm{\mtx{T}} \leq \alpha$.
\end{thm}


\begin{proof}
Assume $\mtx{B}$ has a factorization $\mtx{B} = \mtx{TD}$ with $\norm{\mtx{T}} \leq \alpha$.  We have the chain of implications
\begin{align*}
\mtx{B} = \mtx{TD}
&\quad\Longrightarrow\quad
\enormsq{\mtx{B}\vct{x}} = \enormsq{\mtx{TD}\vct{x}} & \forall \vct{x} \\
&\quad\Longrightarrow\quad
\enormsq{\mtx{B}\vct{x}} \leq \alpha^2 \enormsq{\mtx{D}\vct{x}} & \forall \vct{x} \\
&\quad\Longrightarrow\quad
\vct{x}^\adj \mtx{B}^\adj \mtx{B} \vct{x}
	\leq \alpha^2 \vct{x}^\adj \mtx{D}^2 \vct{x} & \forall \vct{x} \\
&\quad\Longrightarrow\quad
\vct{x}^\adj ( \mtx{B}^\adj \mtx{B} - \alpha^2 \mtx{D}^2 ) \vct{x} \leq 0 & \forall \vct{x} \\
&\quad\Longrightarrow\quad
\mtx{B}^\adj \mtx{B} - \alpha^2 \mtx{D}^2 \psdle \mtx{0},
\end{align*}
where $\psdle$ denotes the semidefinite, or L{\"o}wner, ordering on Hermitian matrices.

Conversely, assume we are provided the inequality
\begin{equation} \label{eqn:nsd-bound}
\mtx{B}^\adj \mtx{B} - \alpha^2 \mtx{D}^2 \psdle \mtx{0}.
\end{equation}
First, we claim that any zero entry in $\mtx{D}$ corresponds with a zero column of $\mtx{B}$.  To check this point, suppose that $d_{jj} = 0$ for an index $j$.  The relation \eqref{eqn:nsd-bound} requires that
$$
0
\geq (\mtx{B}^\adj \mtx{B} - \alpha^2 \mtx{D}^2)_{jj}
= \vct{b}_j^\adj \vct{b}_j.
$$
This inequality is impossible unless $\vct{b}_{j} = \vct{0}$.
To continue, set $\mtx{T} = \mtx{BD}^\psinv$, and observe that $\mtx{B} = \mtx{TD}$ because the zero entries of $\mtx{D}$ correspond with zero columns of $\mtx{B}$.  Therefore, we may factor the diagonal matrix out from \eqref{eqn:nsd-bound} to reach
$$
\mtx{D}( \mtx{T}^\adj \mtx{T} - \alpha^2 \mtx{P} ) \mtx{D}
	\psdle \mtx{0}.
$$
where the matrix $\mtx{P} = \mtx{D}\mtx{D}^\psinv$ is an orthogonal projector.  
Sylvester's theorem on inertia~\cite[Thm.~4.5.8]{HJ85:Matrix-Analysis} ensures that
$
\mtx{T}^\adj \mtx{T} - \alpha^2 \mtx{P} \psdle \mtx{0}.
$
Since $\mtx{P}$ is a projector, this relation implies that
$$
\mtx{T}^\adj \mtx{T} \psdle \alpha^2 \mtx{P} \psdle \alpha^2 \Id.
$$
We conclude that $\norm{\mtx{T}} \leq \alpha$.
\end{proof}

\subsection{Factorization via optimization}
	\label{sec:pietsch-opt}

Recall that the maximum eigenvalue is a convex function on the space of Hermitian matrices, so it can be minimized in polynomial time~\cite{LO96:Eigenvalue-Optimization}.  We are led to consider the convex program
\begin{equation} \label{eqn:pietsch-opt}
\min \ \lambda_{\max}(\mtx{B}^\adj \mtx{B} - \alpha^2 \mtx{F})
\quad\subjto\quad
\text{$\trace(\mtx{F}) = 1$, $\mtx{F}$ diagonal, and $\mtx{F} \geq \mtx{0}$.}
\end{equation}
Owing to Theorem~\ref{thm:pietsch-eig}, there exists a factorization $\mtx{B}=\mtx{TD}$ with $\norm{\mtx{T}} \leq \alpha$ if and only if the value of \eqref{eqn:pietsch-opt} is nonpositive.





Now, if $\mtx{F}$ is a feasible point of \eqref{eqn:pietsch-opt} with a nonpositive objective value, we can factorize
$$
\mtx{B} = \mtx{TD}
\quad\text{with}\quad
\mtx{D} = \mtx{F}^{1/2}, \quad
\mtx{T} = \mtx{BD}^{\psinv}, \quad \text{and}\quad
\norm{\mtx{T}} \leq \alpha.
$$
In fact, it is not necessary to solve \eqref{eqn:pietsch-opt} to optimality.  Suppose $\mtx{B}$ has $s$ columns, and assume we have identified a feasible point $\mtx{F}$ with a (positive) objective value $\eta$.  That is,
$$
\lambda_{\max}(\mtx{B}^\adj\mtx{B} - \alpha^2 \mtx{F}) \leq \eta.
$$
Rearranging this relation, we reach
$$
\lambda_{\max}\left[ \mtx{B}^\adj \mtx{B} - (\alpha^2 + \eta s)
	\widetilde{\mtx{F}} \right] \leq 0
\quad\text{where}\quad
\widetilde{\mtx{F}} = \frac{1}{\alpha^2 + \eta s}(\alpha^2 \mtx{F} + \eta \Id).
$$
Since $\widetilde{\mtx{F}}$ is positive and diagonal with $\trace(\widetilde{\mtx{F}}) = 1$, we obtain the factorization
$$
\mtx{B} = \mtx{TD} \quad\text{with}\quad
\mtx{D} = \widetilde{\mtx{F}}^{1/2}, \quad
\mtx{T} = \mtx{BD}^{-1}, \quad \text{and}\quad
\norm{\mtx{T}} \leq \sqrt{\alpha^2 + \eta s}.
$$

To select a target value for the parameter $\alpha$, we look to the proof of the Kashin--Tzafriri theorem.  If $\mtx{B}$ has $s$ columns, then $\alpha = 8\cnst{K_P}\sqrt{s}$ is an appropriate choice.  Furthermore, since the argument only uses the bound $\norm{\mtx{T}} = \bigO(\sqrt{s})$, it suffices to solve \eqref{eqn:pietsch-opt} with precision $\eta = \bigO(1)$.

\subsection{Other formulations}
	\label{sec:other-forms}
	
In a general setting, a target value for $\alpha$ is not likely to be available.  Let us exhibit an alternative formulation of \eqref{eqn:pietsch-opt} that avoids this inconvenience.
\begin{equation} \label{eqn:pietsch-primal}
\min \ \lambda_{\max}(\mtx{B}^\adj \mtx{B} - \mtx{E}) + \trace(\mtx{E})
\quad\subjto\quad
\text{$\mtx{E}$ diagonal, $\mtx{E} \geq \mtx{0}$.}
\end{equation}
Suppose $\alpha_{\star}$ is the minimal value of $\norm{\mtx{T}}$ achievable in any Pietsch factorization $\mtx{B} = \mtx{TD}$.  It can be shown that $\alpha_{\star}^2$ is the value of \eqref{eqn:pietsch-primal} and that each optimizer $\mtx{E}_{\star}$ satisfies $\trace(\mtx{E}_{\star}) = \alpha_{\star}^2$.
As such, we can construct an optimal Pietsch factorization from a minimizer:
$$
\mtx{B} = \mtx{TD}
\quad\text{with}\quad
\mtx{D} =  (\mtx{E}_{\star}/ \trace(\mtx{E}_{\star}))^{1/2}, \quad
\mtx{T} = \mtx{BD}^\psinv, \quad \text{and}\quad
\norm{\mtx{T}} = \alpha_{\star}.
$$

The dual of \eqref{eqn:pietsch-primal} is the semidefinite program
\begin{equation} \label{eqn:pietsch-dual}
\max \ \ip{ \mtx{B}^\adj \mtx{B} }{ \mtx{Z} }
\quad\subjto\quad
\text{$\diag(\mtx{Z}) = \Id$ and $\mtx{Z} \psdge \mtx{0}$.}
\end{equation}
This is the famous \textsc{maxcut} semidefinite program~\cite{GW95:Improved-Approximation}.  We find an unexpected connection between Pietsch factorization and the problem of partitioning nodes of a graph.

Given a dual optimum, we can easily construct a primal optimum by means of the complementary slackness condition~\cite[Thm.~2.10]{Ali95:Interior-Point-Methods}.  Indeed, each feasible optimal pair $(\mtx{E}_{\star}, \mtx{Z}_{\star})$ satisfies $\mtx{Z}_{\star}(\mtx{B}^\adj \mtx{B} - \mtx{E}_{\star}) = \mtx{0}$.  Examining the diagonal elements of this matrix equation, we find that
$$
\mtx{E}_{\star}
= \diag(\mtx{E}_{\star})
= \diag(\mtx{ZE}_{\star})
= \diag(\mtx{Z}_{\star} \mtx{B}^\adj \mtx{B})
$$
owing to the constraint $\diag(\mtx{Z}_{\star}) = \Id$.  Obtaining a dual optimum from a primal optimum, however, requires more ingenuity.

\begin{rem} \label{rem:inf2-norm}
According to Theorem~\ref{thm:pietsch} and the discussion here, the optimal value of \eqref{eqn:pietsch-primal} overestimates $\pnorm{\infty\to 2}{\mtx{B}}^2$ by a multiplicative factor no greater than $\cnst{K_P}^2$.  As a result, the optimization problem~\eqref{eqn:pietsch-primal} can be used to design an approximation algorithm for $(\infty, 2)$ norms.
\end{rem}

\subsection{Algorithmic aspects}

The purpose of this paper is not to rehash methods for solving a standard optimization problem, so we keep this discussion brief.  It is easy to see that \eqref{eqn:pietsch-opt} can be framed as a (nonsmooth) convex optimization over the probability simplex.  Appendix~\ref{app:emd} outlines an elegant technique, called Entropic Mirror Descent \cite{BT03:Mirror-Descent}, designed specifically for this class of problems.
Although the EMD algorithm is (theoretically) not the most efficient approach to \eqref{eqn:pietsch-opt}, preliminary experiments suggest that its empirical performance rivals more sophisticated techniques.

For a concrete time bound, we refer to Alizadeh's work on primal--dual potential reduction methods for semidefinite programming \cite{Ali95:Interior-Point-Methods}.  When $\mtx{B}$ has dimension $m \times s$, the cost of forming $\mtx{B}^\adj\mtx{B}$ is at most $\bigO(s^2 m)$.  Then the cost of solving \eqref{eqn:pietsch-dual} is no more than $\widetilde{\bigO}( s^{3.5} )$, where the tilde indicates that log-like factors are suppressed.

%


\section{An Algorithm for Kashin--Tzafriri}
	\label{sec:kt-alg}

At this point, we have amassed the mat{\'e}riel necessary to deploy an algorithm that constructs the set $\tau$ promised by the Kashin--Tzafriri theorem.  The procedure appears on page~\pageref{alg:kt} as Algorithm~\ref{alg:kt}.  The following result describes its performance.

\begin{thm} \label{thm:kt-alg}
Suppose $\mtx{A}$ is an $m \times n$ standardized matrix.  With probability at least $4/5$, Algorithm~\ref{alg:kt} produces a set $\tau = \tau_{\star}$ of column indices for which
$$
\abs{\tau} \geq \frac{1}{2} \strank(\mtx{A})
\quad\text{and}\quad
\norm{\mtx{A}_{\tau}} \leq 15.
$$
The computational cost is bounded by $\widetilde{\bigO}( \abs{\tau}^2 m + \abs{\tau}^{3.5})$.
\end{thm}

Remarkably, Algorithm~\ref{alg:kt} is sublinear in the size of the matrix when $\strank(\mtx{A}) = {\rm o}(n^{1/3.5})$.  Better methods for solving~\eqref{eqn:pietsch-opt} would strengthen this bound.

\begin{proof}
According to Section~\ref{sec:kt-pf}, the procedure {\sc Norm-Reduce} has failure probability less than $7/8$ when $s \leq 2 \strank(\mtx{A})$.  The probability the inner loop fails to produce an acceptable set $\tau_{\star}$ of size $s/2$ is at most $(7/8)^{8\log_2(s)}$.  So the probability the algorithm fails before $s \geq \strank(\mtx{A})$ is at most
$$
\sum\nolimits_{j=2}^\infty (7/8)^{8j} =
	\frac{(7/8)^{16}}{1 - (7/8)^8} < 0.2.
$$
With constant probability, we obtain a set $\tau_{\star}$ with cardinality at least $\strank(\mtx{A})/2$.


The cost of the procedure {\sc Norm-Reduce} is dominated by the cost of the Pietsch factorization, which is $\widetilde{\bigO}(s^2m + s^{3.5})$ for a fixed $s$.  Summing over $s$ and $k$, we find that the total cost of all the invocations of {\sc Norm-Reduce} is dominated (up to logarithmic factors) by the cost of the final invocation, during which the parameter $s \leq 2 \abs{\tau_{\star}}$.

An estimate of the spectral norm of $\mtx{A}_{\tau}$ can be obtained as a by-product of solving~\eqref{eqn:pietsch-opt}.  Indeed, Proposition~\ref{prop:pietsch-cols} and the discussion in Section~\ref{sec:pietsch-opt} show that we can bound the spectral norm in terms of the parameter $\alpha$ and the objective value obtained in \eqref{eqn:pietsch-opt}.
\end{proof}



\section{The Bourgain--Tzafriri Theorem}
	\label{sec:bt}


Our proof of the Bourgain--Tzafriri theorem is almost identical in structure with the proof of the Kashin--Tzafriri theorem.  This streamlined argument appears to be simpler than all previously published approaches, but it contains no significant conceptual innovations.  Our discussion culminates in an algorithm remarkably similar to Algorithm~\ref{alg:kt}.

\subsection{Preliminary results}

Suppose $\mtx{A}$ is a standardized matrix with $n$ columns.  We will work instead with a related matrix $\mtx{H} = \mtx{A}^\adj \mtx{A} - \Id$, which is called the \term{hollow Gram matrix}.  The advantage of considering the hollow Gram matrix is that we can perform column selection on $\mtx{A}$ simply by reducing the norm of $\mtx{H}$.

\begin{prop} \label{prop:H-A}
Suppose $\mtx{A}$ is a standardized matrix with hollow Gram matrix $\mtx{H}$.  If $\tau$ is a set of column indices for which $\norm{\mtx{H}_{\tau \times \tau}} \leq 0.5$, then $\kappa(\mtx{A}_{\tau}) \leq \sqrt{3}$.
\end{prop}

\begin{proof}
The hypothesis $\norm{\mtx{H}_{\tau \times \tau}} \leq 0.5$ implies that the eigenvalues of $\mtx{H}_{\tau \times \tau}$ lie in the range $[-0.5, 0.5]$.  Since $\mtx{H}_{\tau \times \tau} = \mtx{A}_{\tau}^\adj \mtx{A}_{\tau} - \Id$, the eigenvalues of $\mtx{A}_{\tau}^\adj \mtx{A}_\tau$ fall in the interval $[0.5, 1.5]$.  An equivalent condition is that
$0.5 \leq \enormsq{ \mtx{A}_{\tau} \vct{x} } \leq 1.5$ whenever $\enorm{\vct{x}} = 1$.
We conclude that
$$
\kappa(\mtx{A}_\tau) = \max\left\{
	\frac{\enorm{\mtx{A}_\tau \vct{x}}}{\enorm{\mtx{A}_{\tau} \vct{y}}} :
	\enorm{\vct{x}} = \enorm{\vct{y}} = 1 \right\}
	\leq \sqrt{\frac{1.5}{0.5}} = \sqrt{3}.
$$
Thus, a norm bound for $\mtx{H}_{\tau \times \tau}$ yields a condition number bound for $\mtx{A}_{\tau}$.
\end{proof}

As we mentioned before, random selection may reduce other norms even if it does not reduce the spectral norm.  Define the natural norm on linear maps from $\ell_\infty$ to $\ell_1$ by the formula
$$
\pnorm{\infty\to1}{\mtx{G}}
	= \max\{ \pnorm{1}{ \mtx{G}\vct{x} } : \infnorm{\vct{x}} = 1 \}.
$$
This norm is closely related to the cut norm, which plays a starring role in graph theory~\cite{AN04:Approximating-Cut}.  For a general $s \times s$ matrix $\mtx{G}$, the best inequality between the $(\infty,1)$ norm and the spectral norm is $\pnorm{\infty\to1}{\mtx{G}} \leq s \norm{\mtx{G}}$.  Rohn \cite{Roh00:Computing-Norm} has established that there is a class of positive semidefinite, integer matrices for which it is {\sf NP}-hard to determine the $(\infty,1)$ norm within an absolute tolerance of 1/2.  Nevertheless, it can be approximated within a small relative factor in polynomial time~\cite{AN04:Approximating-Cut}.

The $(\infty, 1)$ norm decreases when we randomly sample a principal submatrix.  The following result, which we establish in Appendix~\ref{app:inf1}, is a direct consequence of Rudelson and Vershynin's work on the cut norm of random submatrices~\cite[Thm.~1.5]{RV07:Sampling-Large}.

\begin{thm} \label{thm:inf1-reduce}
Suppose $\mtx{A}$ is an $n$-column standardized matrix with hollow Gram matrix $\mtx{H}$.  Choose
$$
s \leq \lceil \cnst{c} \cdot \strank(\mtx{A}) \rceil,
$$
and draw a uniformly random subset $\sigma$ with cardinality $s$ from $\{1,2, \dots, n\}$.  Then
$$
\Expect \pnorm{\infty\to1}{\mtx{H}_{\sigma \times \sigma}}
	\leq \frac{s}{9}.
$$
In particular, $\pnorm{\infty\to1}{\mtx{H}_{\sigma \times \sigma}} \leq s/8$ with probability at least $1/9$.
\end{thm}

To connect the $(\infty,1)$ norm with the spectral norm, 
we call on the celebrated factorization of Grothendieck~\cite[p.~56]{Pis86:Factorization-Linear}.

\begin{thm}[Grothendieck Factorization] \label{thm:groth-fact}
Each matrix $\mtx{G}$ can be factored as $\mtx{G} = \mtx{D}_1 \mtx{T} \mtx{D}_2$ where
\begin{enumerate}
\item	$\mtx{D}_i$ is a nonnegative, diagonal matrix with $\trace(\mtx{D}_i^2) = 1$ for $i = 1,2$, and
\item	$\pnorm{\infty\to1}{\mtx{G}} \leq \norm{\mtx{T}} \leq \cnst{K_G} \pnorm{\infty\to1}{\mtx{G}}$.
\end{enumerate}
When $\mtx{G}$ is Hermitian, we may take $\mtx{D}_1 = \mtx{D}_2$.
\end{thm}

The precise value of the Grothendieck constant $\cnst{K_G}$ remains an outstanding open question, but it is known to depend on the scalar field~\cite[Sec.~5e]{Pis86:Factorization-Linear}.
\begin{itemize}
\item	When the scalar field is real, $1.570 \leq \pi/2 \leq \cnst{K_G}(\Rspace{}) \leq \pi/(2\log(1+\sqrt{2})) \leq 1.783$.

\item	When the scalar field is complex, $1.338 \leq \cnst{K_G}(\Cspace{}) \leq 1.405$.
\end{itemize}
For positive semidefinite $\mtx{G}$, the real (resp., complex) Grothendieck constant equals the square of the real (resp., complex) Pietsch constant because $\pnorm{\infty\to1}{\mtx{B}^\adj \mtx{B}} = \pnorm{\infty\to2}{\mtx{B}}^2$.

The following proposition describes the role of the Grothendieck factorization in the selection of submatrices with controlled spectral norm.

\begin{prop} \label{prop:groth-col}
Suppose $\mtx{G}$ is an $s \times s$ Hermitian matrix.  There is a set $\tau$ of column indices for which
$$
\abs{\tau} \geq \frac{s}{2}
\quad\text{and}\quad
\norm{\mtx{G}_{\tau \times \tau}} \leq
	\frac{2 \cnst{K_G}}{s} \pnorm{\infty\to1}{\mtx{G}}.
$$
\end{prop}

\begin{proof}
Consider a Grothendieck factorization $\mtx{G} = \mtx{DTD}$, and identify $\tau = \{ j : d_{jj}^2 \leq s/2 \}$.  The remaining details echo the proof of Proposition~\ref{prop:pietsch-cols}.
\end{proof}

\subsection{Proof of Bourgain--Tzafriri}

Suppose $\mtx{A}$ is a standardized matrix with $n$ columns, and consider its hollow Gram matrix $\mtx{H}$.  Theorem~\ref{thm:inf1-reduce} provides a set $\sigma$ for which
$$
\abs{\sigma} \geq \cnst{c} \cdot \strank(\mtx{A})
\quad\text{and}\quad
\pnorm{\infty\to1}{\mtx{H}_{\sigma \times \sigma}} \leq \frac{s}{8}.
$$
Apply Proposition~\ref{prop:groth-col} to the $s \times s$ matrix $\mtx{G} = \mtx{H}_{\sigma\times\sigma}$ to obtain a further subset $\tau$ inside $\sigma$ with
$$
\abs{\tau} \geq \frac{s}{2}
\quad\text{and}\quad
\norm{\mtx{G}_{\tau \times \tau}}
	\leq \frac{2\cnst{K_G}}{s} \pnorm{\infty\to1}{\mtx{G}}.
$$
Since $2 \cnst{K_G} < 4$ and $\mtx{H}_{\tau \times \tau} = \mtx{G}_{\tau\times\tau}$, we determine that
$$
\abs{\tau} \geq \frac{\cnst{c}}{2} \cdot \strank(\mtx{A})
\quad\text{and}\quad
\norm{\mtx{H}_{\tau \times \tau}} \leq 0.5.
$$
In view of Proposition~\ref{prop:H-A}, we conclude $\kappa(\mtx{A}_{\tau}) \leq \sqrt{3}$.

Now, take another step back and notice that this here argument is nearly algorithmic.  The random selection of $\sigma$ can easily be implemented in practice, even though the proof does not specify the value of $\cnst{c}$.  Given a Grothendieck factorization $\mtx{G} = \mtx{DTD}$, it is straightforward to identify the subset $\tau$.  The challenge, as before, is to produce the factorization.  

\subsection{Grothendieck factorization via convex optimization}

As with the Pietsch factorization, the Grothendieck factorization can be identified from the solution to a convex program.

\begin{thm} \label{thm:groth-eig}
Suppose $\mtx{G}$ is Hermitian.  The factorization $\mtx{G} = \mtx{DTD}$ satisfies $\norm{\mtx{T}} \leq \alpha$ if and only if $\mtx{D}$ satisfies
\begin{equation} \label{eqn:groth-nsd}
\lambda_{\max}
\begin{bmatrix}
- \alpha \mtx{D}^2 & \mtx{G} \\
\mtx{G} & -\alpha \mtx{D}^2
\end{bmatrix}
\leq 0.
\end{equation}
In particular, if no $\mtx{D}$ verifies this bound, then no factorization $\mtx{G} = \mtx{DTD}$ admits $\norm{\mtx{T}} \leq \alpha$.
\end{thm}

\begin{proof}
To check the forward implication, we essentially repeat the argument we used in Theorem~\ref{thm:pietsch-eig} for the Pietsch case.  This reasoning yields the pair of relations
$$
\mtx{G} - \alpha \mtx{D}^2 \psdle \mtx{0}
\quad\text{and}\quad
- \mtx{G} - \alpha \mtx{D}^2 \psdle \mtx{0}.
$$
Together, these two relations are equivalent with \eqref{eqn:groth-nsd} because
$$
\begin{bmatrix}
- \alpha \mtx{D}^2 & \mtx{G} \\ \mtx{G} & -\alpha \mtx{D}^2
\end{bmatrix}
= \frac{1}{2}
\begin{bmatrix} \Id & \Id \\ -\Id & \Id \end{bmatrix}^\adj
\begin{bmatrix}
	\mtx{G} - \alpha \mtx{D}^2 & \\ & - \mtx{G} - \alpha \mtx{D}^2 \end{bmatrix}
\begin{bmatrix} \Id & \Id \\ -\Id & \Id \end{bmatrix}.
$$

To prove the reverse implication, we assume that~\eqref{eqn:groth-nsd} holds.  First, we must check that $d_{jj} = 0$ implies that $\vct{g}_j = \vct{0}$.  To verify this claim, observe that
$$
0 \geq
\begin{bmatrix} \alpha \\ \vct{g}_j \end{bmatrix}^\adj
\begin{bmatrix} 0 & \vct{g}_j^\adj \\ \vct{g}_j & -\alpha\mtx{D}^2 \end{bmatrix}
\begin{bmatrix} \alpha \\ \vct{g}_j \end{bmatrix}
=  \alpha \left( 2 \enormsq{\vct{g}_j} - \vct{g}_j^\adj \mtx{D}^2 \vct{g}_j \right)
\geq \alpha \enormsq{\vct{g}_j}
$$
because $\trace(\mtx{D}^2) = 1$.  Therefore, we may construct a Grothendieck factorization $\mtx{G} = \mtx{DTD}$ with $\norm{\mtx{T}} \leq \alpha$ by setting $\mtx{T} = \mtx{D}^\psinv \mtx{G} \mtx{D}^\psinv$.
\end{proof}




This discussion leads us to frame the eigenvalue minimization problem
\begin{equation} \label{eqn:groth-opt}
\min \ \lambda_{\max}
\begin{bmatrix}
- \alpha \mtx{F} & \mtx{G} \\
\mtx{G} & -\alpha \mtx{F}
\end{bmatrix}
\quad\subjto\quad
\text{$\trace(\mtx{F}) = 1$, $\mtx{F}$ diagonal, $\mtx{F} \geq \mtx0$}
\end{equation}
Owing to Theorem~\ref{thm:groth-eig}, there is a factorization $\mtx{G} = 
\mtx{DTD}$ with $\norm{\mtx{T}} \leq \alpha$ if and only if the value of \eqref{eqn:groth-opt} is nonpositive.

As in Section~\ref{sec:pietsch-opt}, we can easily construct Grothendieck factorizations from (imprecise) solutions to the problem \eqref{eqn:groth-opt}.
The proof of Bourgain--Tzafriri suggests that an appropriate value for the parameter $\alpha = s/4$.  Furthermore, we do not need to solve \eqref{eqn:groth-opt} to optimality to obtain the required information.  Indeed, it suffices to produce a feasible point with an objective value of $\bigO(1)$.  



To solve \eqref{eqn:groth-opt} in practice, we again propose the Entropic Mirror Descent algorithm \cite{BT03:Mirror-Descent}.  Appendix~\ref{app:emd} describes the application to this problem.  To provide a concrete bound on the computational cost, we remark that, when $\mtx{A}_{\tau}$ has dimension $m \times s$, forming $\mtx{G} = \mtx{A}_{\tau}^\adj \mtx{A}_{\tau} - \Id$ costs at most $\bigO(s^2 m)$, and Alizadeh's interior-point method~\cite{Ali95:Interior-Point-Methods} requires $\widetilde{\bigO}(s^{3.5})$ time.

\begin{rem} \label{rem:inf1-norm}
For symmetric $\mtx{G}$, Theorem~\ref{thm:groth-fact} shows that the norm $\pnorm{\infty\to1}{\mtx{G}}$ is approximated within a factor $\cnst{K_G}$ by the least $\alpha$ for which \eqref{eqn:groth-opt} has a nonpositive value.  A natural reformulation of \eqref{eqn:groth-opt} can identify this value of $\alpha$ automatically (cf.~Section~\ref{sec:other-forms}).  For nonsymmetric $\mtx{G}$, similar optimization problems arise.  These ideas yield new approximation algorithms for the $(\infty, 1)$ norm.   
\end{rem}


\subsection{An algorithm for Bourgain--Tzafriri}

We are prepared to state our algorithm for producing the set $\tau$ described by the Bourgain--Tzafriri theorem.  The procedure appears as Algorithm~\ref{alg:bt} on page~\pageref{alg:bt}.  Note the striking similarity with Algorithm~\ref{alg:kt}.  The following result describes the performance of the algorithm.  We omit the proof, which parallels that of Theorem~\ref{thm:kt-alg}.

\begin{thm}
Suppose $\mtx{A}$ is an $m \times n$ standardized matrix.  With probability at least $3/4$, Algorithm~\ref{alg:bt} produces a set $\tau = \tau_{\star}$ of column indices for which
$$
\abs{\tau} \geq \cnst{c} \cdot \strank(\mtx{A})
\quad\text{and}\quad
\kappa(\mtx{A}_{\tau}) \leq \sqrt{3}.
$$
The computational cost is bounded by $\widetilde{\bigO}( \abs{\tau}^2 m + \abs{\tau}^{3.5})$.
\end{thm}


\section{Future Directions} \label{sec:discussion}

After the initial work~\cite{BT87:Invertibility-Large}, additional research has clarified the role of the stable rank.  We highlight a positive result of Vershynin~\cite[Cor.~7.1]{Ver01:Johns-Decompositions} and a negative result of Szarek~\cite[Thm.~1.2]{Sza90:Spaces-Large} which together imply that the stable rank describes \emph{precisely} how large a well-conditioned column submatrix can in general exist.  See~\cite[Sec.~5]{Ver01:Johns-Decompositions} for a more detailed discussion.

\begin{thm}[Vershynin 2001] \label{thm:versh}
Fix $\eps > 0$.  For each matrix $\mtx{A}$, there is a set $\tau$ of column indices for which
$$
\abs{\tau} \geq (1 - \eps) \cdot \strank(\mtx{A})
\quad\text{and}\quad
\kappa(\mtx{A}_{\tau}) \leq \cnst{C}(\eps).
$$
\end{thm}

\begin{thm}[Szarek]
There is a sequence $\{\mtx{A}(n)\}$ of matrices of increasing dimension for which
$$
\abs{\tau} = \strank(\mtx{A})
\quad\Longrightarrow\quad
\kappa( \mtx{A}_{\tau} ) = \omega(1).
$$
\end{thm}

Vershynin's proof constructs the set $\tau$ in Theorem~\ref{thm:versh} with a complicated iteration that interleaves the Kashin--Tzafriri theorem and the Bourgain--Tzafriri theorem.  We believe that the argument can be simplified substantially and developed into a column selection algorithm.  This achievement might lead to a new method for performing rank-revealing factorizations, which could have a significant impact on the practice of numerical linear algebra.


%

\newpage

\begin{algorithm}[htb]
\caption{Constructive version of Kashin--Tzafriri Theorem} \label{alg:kt}
\centering \fbox{
\begin{minipage}{.9\textwidth} 
\vspace{4pt}
\algname{KT}{$\mtx{A}$}
\alginout{Standardized matrix $\mtx{A}$ with $n$ columns}
{A subset $\tau_{\star}$ of $\{1, 2, \dots, n\}$}
\algdescript{Produces $\tau_{\star}$ such that $\abs{\tau_{\star}} \geq \strank(\mtx{A})/2$ and $\norm{\vct{A}_\tau} \leq 15$ w.p.~4/5}
\vspace{6pt}

\begin{algtab}
$\tau_{\star} = \{1\}$ \\
\algfor{$s = 4, 8, 16, \dots, n$}
	\algfor{$k = 1, 2, 3, \dots, 8 \log_2 s$}
		$\tau = \textsc{Norm-Reduce}(\mtx{A}, s)$ \\
		\algifthen{$\norm{\mtx{A}_{\tau}} \leq 15$}
			{$\tau_{\star} = \tau$ and {\bf break}}
	\algend
	\algifthen{$\abs{\tau_{\star}} < s$}{{\bf exit}}
\algend
\end{algtab}

\vspace{10pt} \hrule \vspace{10pt}

\algname{Norm-Reduce}{$\mtx{A}$, $s$}
\alginout{Standardized matrix $\mtx{A}$ with $n$ columns, a parameter $s$}
{A subset $\tau$ of $\{1, 2, \dots, n\}$}
\vspace{6pt}

\begin{algtab}
Draw a uniformly random set $\sigma$ with cardinality $s$ from $\{1,2,\dots, n\}$ \\
Solve \eqref{eqn:pietsch-opt} with $\mtx{B} = \mtx{A}_{\sigma}$ and $\alpha = 8\cnst{K_P}\sqrt{s}$ to obtain a 
	factorization $\mtx{B} = \mtx{TD}$  \\
Return $\tau = \{ j \in \sigma : d_{jj}^2 \leq 2/s \}$
\end{algtab}
\end{minipage}}
\end{algorithm}

\begin{algorithm}[htb]
\caption{Constructive version of Bourgain--Tzafriri Theorem} \label{alg:bt}
\centering \fbox{
\begin{minipage}{.9\textwidth} 
\vspace{4pt}
\algname{BT}{$\mtx{A}$}
\alginout{Standardized matrix $\mtx{A}$ with $n$ columns}
{A subset $\tau_{\star}$ of $\{1, 2, \dots, n\}$}
\algdescript{Produces $\tau_{\star}$ such that $\abs{\tau_{\star}} \geq \strank(\mtx{A})/2$ and $\kappa(\vct{A}_\tau) \leq \sqrt{3}$ w.p.~3/4}
\vspace{6pt}

\begin{algtab}
$\tau_{\star} = \{1\}$ \\
\algfor{$s = 4, 8, 16, \dots, n$}
	\algfor{$k = 1, 2, 3, \dots, 8 \log_2 s$}
		$\tau = \textsc{Cond-Reduce}(\mtx{A}, s)$ \\
		\algifthen{$\kappa(\mtx{A}_{\tau}) \leq \sqrt{3}$}
			{$\tau_{\star} = \tau$ and {\bf break}}
	\algend
	\algifthen{$\abs{\tau_{\star}} < s$}{{\bf exit}}
\algend
\end{algtab}

\vspace{10pt} \hrule \vspace{10pt}

\algname{Cond-Reduce}{$\mtx{A}$, $s$}
\alginout{Standardized matrix $\mtx{A}$ with $n$ columns, a parameter $s$}
{A subset $\tau$ of $\{1, 2, \dots, n\}$}
\vspace{6pt}

\begin{algtab}
Draw a uniformly random set $\sigma$ with cardinality $s$ from $\{1,2,\dots, n\}$ \\
Solve \eqref{eqn:groth-opt} with $\mtx{G} = \mtx{A}_{\sigma}^\adj \mtx{A}_{\sigma} - \Id$ and $\alpha = s/4$ to obtain factorization $\mtx{G} = \mtx{DTD}$  \\
Return $\tau = \{ j \in \sigma : d_{jj}^2 \leq 2/s \}$
\end{algtab}
\end{minipage}}
\end{algorithm}



\clearpage

\newpage

\appendix

\section{Random Reduction of Norms}
	\label{app:rdm-norm}
	
How does the norm of a matrix change when we pass to a random submatrix?  This question has great importance in modern functional analysis, but it also has implications for the design of algorithms.  This appendix describes some general results on how random selection reduces the $(\infty, 2)$ norm and the $(\infty, 1)$ norm.  We also specialize these results to the structured matrices that appear in the proofs of Theorem~\ref{thm:kt} and Theorem~\ref{thm:bt}.


\subsection{Random Coordinate Models}

We begin with two standard models for selecting random submatrices, and we describe how these models are related for an important class of matrix norms.

A matrix norm is \term{monotonic} if the norm of a matrix exceeds the norm of every (rectangular) submatrix.  More precisely, the norm $\triplenorm{\cdot}$ is monotonic if
$$
\triplenorm{ \mtx{PAP}' } \leq \triplenorm{\mtx{A}}
$$
for each matrix $\mtx{A}$ and each pair $\mtx{P}, \mtx{P}'$ of diagonal (i.e., coordinate) projectors.  The basic example of a monotonic matrix norm is the natural norm on operators from $\ell_p$ to $\ell_q$ with $p, q$ in $[1,\infty]$, which is defined as
$$
\pnorm{p\to q}{ \mtx{A} } =
	\max\{ \pnorm{q}{\mtx{A}\vct{x}} : \pnorm{p}{\vct{x}} = 1 \}.
$$

Fix a number $\delta$ in $[0,1]$, and denote by $\mtx{P}_{\delta}$ a random $n \times n$ diagonal matrix where exactly $s = \lfloor \delta n \rfloor$ entries equal one and the rest equal zero.  This matrix can be viewed as a projector onto a random set of $s$ coordinates.  Therefore, we may treat $\mtx{A}\mtx{P}_{\delta}$ as a random $s$-column submatrix of $\mtx{A}$ by ignoring the zeroed columns.  Although this model is conceptually appealing, it can be difficult to analyze because of the dependencies among coordinates.

Let us introduce a simpler model for selecting random coordinates.  We denote by $\mtx{R}_{\delta}$ a random $n \times n$ diagonal matrix whose entries are independent 0--1 random variables with common mean $\delta$.  This matrix is a projector onto a random set of coordinates with \emph{average} cardinality $\delta n$.

There is a basic result connecting these two models.  The statement here follows directly from the argument in \cite[Lem.~14]{Tro08:Linear-Independence}.

\begin{prop}[Poissonization] \label{prop:poisson}
Let $\triplenorm{\cdot}$ be a monotonic matrix norm.  For each matrix $\mtx{A}$ with $n$ columns, it holds that
$$
\Expect \triplenorm{ \mtx{A} \mtx{P}_{\delta} }
	\leq 2 \Expect \triplenorm{ \mtx{A} \mtx{R}_{\delta} }.
$$
For each $n \times n$ matrix $\mtx{H}$, it holds that
$$
\Expect \triplenorm{ \mtx{P}_{\delta} \mtx{H} \mtx{P}_{\delta} }
	\leq 2 \Expect \triplenorm{ \mtx{R}_{\delta} \mtx{H} \mtx{R}_{\delta} }.
$$
\end{prop}



\subsection{Reduction of the $(\infty,2)$ norm}
	\label{app:inf2}

We begin with a general result on the $(\infty, 2)$ norm of a uniformly random set of columns drawn from a fixed matrix.  The basic argument appears already in the work of Bourgain and Tzafriri~\cite[Thm.~1.1]{BT91:Problem-Kadison-Singer}, but modern proofs are a little simpler.  (See~\cite[Lem.~2.3]{Ver05:Random-Sets}, for example.)  The version here offers especially good constants.


\begin{thm} \label{thm:inf2-reduce}
Fix $\delta \in [0, 1]$, and suppose $\mtx{A}$ is a matrix with $n$ columns. Then
$$
\Expect \pnorm{\infty\to2}{\mtx{A}\mtx{R}_{\delta}}
	\leq \sqrt{2\delta(1-\delta)} \fnorm{\mtx{A}}
		+ \delta \pnorm{\infty\to2}{\mtx{A}}.
$$
\end{thm}

We postpone the argument to the next section so we may note a corollary that appears as a key step in the proof of the Kashin--Tzafriri theorem.

\begin{cor} \label{cor:inf2-kt}
Suppose $\mtx{A}$ is a standardized matrix with $n$ columns.  Choose $s \leq \lceil 2 \strank(\mtx{A}) \rceil$, and write $\delta = s/n$.  Then
$$
\Expect \pnorm{\infty\to2}{\mtx{A}\mtx{P}_{\delta}}
	\leq 7 \sqrt{s}.
$$
\end{cor}

\begin{proof}
Owing to the standardization, $1 \leq \strank(\mtx{A}) = n /\normsq{\mtx{A}}$.  It follows that
$$
\delta \leq \frac{2 \strank(\mtx{A}) + 1}{n}
	\leq \frac{3 \strank(\mtx{A})}{n}
	= \frac{3}{\normsq{\mtx{A}}}.
$$
Apply the Poissonization result, Proposition~\ref{prop:poisson}, to see that
$$
\Expect \pnorm{\infty\to2}{\mtx{A}\mtx{P}_{\delta}}
	\leq 2 \Expect \pnorm{\infty\to2}{\mtx{A}\mtx{R}_{\delta}}.
$$
Theorem~\ref{thm:inf2-reduce} yields
$$
\Expect \pnorm{\infty\to2}{\mtx{A}\mtx{P}_{\delta}}
	\leq 2\sqrt{2\delta} \fnorm{\mtx{A}}
		+ 2\delta \pnorm{\infty\to2}{\mtx{A}}. 
$$
Since $\mtx{A}$ has $n$ unit-norm columns, it holds that $\fnorm{\mtx{A}} = \sqrt{n}$.  We also have the general bound $\pnorm{\infty\to2}{ \mtx{A} } \leq \sqrt{n} \norm{\mtx{A}}$.  Therefore,
$$
\Expect \pnorm{\infty\to2}{\mtx{A}\mtx{P}_{\delta}}
	\leq 2\sqrt{2\delta n} + 2 \delta \sqrt{n} \norm{\mtx{A}}
	= 2 \sqrt{s} \left[ \sqrt{2} + \sqrt{\delta} \norm{\mtx{A}} \right].
$$
Introduce the bound on $\delta$ and make a numerical estimate to complete the proof. 
\end{proof}

\subsection{Proof of Theorem~\ref{thm:inf2-reduce}}

We must bound the quantity
$$
E = \Expect\pnorm{\infty\to2}{\mtx{A}\mtx{R}_{\delta}}.
$$
It turns out that it is easier to work with the $(2,1)$ norm, which is dual to the $(\infty,2)$ norm, because there are some special methods that apply.  Rewrite the expression as
$$
E = \Expect \pnorm{2\to 1}{\mtx{R}_{\delta} \mtx{A}^\adj}
	= \Expect \max_{\enorm{\vct{x}} = 1} \sum\nolimits_{j=1}^n
	\delta_j \absip{ \vct{a}_j }{ \vct{x} }
$$
where $\{\delta_j\}$ is a sequence of independent 0--1 random variables with common mean $\delta$.
In the sequel, we simplify notation by omitting the restriction on the vector $\vct{x}$ and the limits from the sum.

The next step is to center and symmetrize the selectors.  First, add and subtract the mean of each term from the sum and use the subadditivity of the maximum to obtain
\begin{align*}
E &\leq \Expect \max_{\vct{x}}
	\sum\nolimits_j (\delta_j - \delta) \absip{\vct{a}_j}{\vct{x}}
	+ \max_{\vct{x}} \sum\nolimits_j \delta \absip{\vct{a}_j}{\vct{x}} \\
	&= \Expect \max_{\vct{x}}
	\sum\nolimits_j (\delta_j - \delta) \absip{\vct{a}_j}{\vct{x}}
		+ \delta \pnorm{2 \to 1}{ \mtx{A}^\adj } \\
	&= \Expect \max_{\vct{x}}
	\sum\nolimits_j (\delta_j - \delta) \absip{\vct{a}_j}{\vct{x}}
		+ \delta \pnorm{ \infty \to 2}{ \mtx{A} }.
\end{align*}
We focus on the first term, which we abbreviate by the letter $F$.  Let $\{\delta_j'\}$ be an independent copy of the sequence $\{\delta_j\}$.  Jensen's inequality allows that
\begin{align*}
F &= \Expect \max_{\vct{x}} \sum\nolimits_j
	(\delta_j - \Expect \delta_j') \absip{\vct{a}_j}{\vct{x}} \\
	&\leq \Expect \max_{\vct{x}} \sum\nolimits_j
	(\delta_j - \delta_j') \absip{\vct{a}_j}{\vct{x}}.
\end{align*}
Observe that $\{ \delta_j - \delta_j' \}$ is a sequence of independent, symmetric random variables.  Thus, we may multiply each one by a random sign without changing the expectation~\cite[Lem.~6.3]{LT91:Probability-Banach}.  That is,
$$
F \leq \Expect \max_{\vct{x}} \sum\nolimits_j
	\eps_j (\delta_j - \delta_j') \absip{\vct{a}_j}{\vct{x}}
$$
where $\{\eps_j\}$ is a sequence of independent Rademacher (i.e., uniform $\pm 1$) random variables.  


Now, we invoke a specific type of Rademacher comparison~\cite[Thm.~4.12 et seq.]{LT91:Probability-Banach} to remove the absolute values from the inner product:
$$
F \leq \Expect \max_{\vct{x}} \sum\nolimits_j
	\eps_j (\delta_j - \delta_j') \ip{\vct{a}_j}{\vct{x}}
	= \Expect \max_{\vct{x}} \ip{
		\sum\nolimits_j \eps_j (\delta_j - \delta_j') \vct{a}_j }{ \vct{x} }.
$$
Since $\vct{x}$ ranges over the $\ell_2$ unit sphere, we reach
$$
F \leq \Expect \enorm{ \sum\nolimits_j \eps_j (\delta_j - \delta_j') \vct{a}_j }.
$$

The remaining expectations are elementary.  First, apply H{\"o}lder's inequality to obtain
$$
F \leq \left( \Expect \enormsq{ \sum\nolimits_j
	\eps_j (\delta_j - \delta_j') \vct{a}_j } \right)^{1/2}.
$$
Compute the expectation with respect to $\{\eps_j\}$ and then with respect to $\{\delta_j \}$ and $\{\delta_j'\}$.
$$
F \leq \left( \Expect \sum\nolimits_j
	(\delta_j - \delta_j')^2 \enormsq{\vct{a}_j} \right)^{1/2}
	= \left( \sum\nolimits_j 2\delta(1 - \delta) \enormsq{\vct{a}_j} \right)^{1/2}
	= \sqrt{2 \delta(1-\delta)} \fnorm{\mtx{A}}.
$$
Introduce this bound on $F$ into the bound on $E$ to conclude that
$$
\Expect \pnorm{\infty\to2}{\mtx{A}\mtx{P}_{\delta}}
	\leq \sqrt{2\delta(1-\delta)} \fnorm{\mtx{A}} + \delta \pnorm{\infty\to2}{\mtx{A}}.
$$
This is the advertised estimate.

\subsection{Reduction of the $(\infty, 1)$ norm}	\label{app:inf1}

The impact of random selection on the $(\infty,1)$ norm has already received some attention in the theoretical computer science literature because of a connection with graph cuts.  The following result of Rudelson and Vershynin contains detailed information on the $(\infty, 1)$ norm of a random principal submatrix.  The statement involves an auxiliary norm
$$
\pnorm{\rm col}{\mtx{H}} = \sum\nolimits_j \enorm{\mtx{H}\onevct_j},
$$
where $\{\onevct_j\}$ is the set of standard basis vectors.  In words, we sum the Euclidean norms of the columns of the matrix.

\begin{thm}[Rudelson--Vershynin] \label{thm:rv-inf1}
Fix $\delta \in [0,1]$, and suppose $\mtx{H}$ is an $n \times n$ matrix.  Then
$$
\Expect \pnorm{\infty\to1}{\mtx{R}_\delta \mtx{H} \mtx{R}_{\delta}}
	\leq \cnst{C} \left[ \delta^2 \pnorm{\infty\to1}{\mtx{H} - \diag(\mtx{H})}
		+ \delta^{3/2} \left(\pnorm{\rm col}{\mtx{H}}
	+ \pnorm{\rm col}{\mtx{H}^\adj} \right)
		+ \delta \pnorm{\infty\to1}{ \diag(\mtx{H}) } \right].
$$
\end{thm}

Theorem~\ref{thm:rv-inf1} is established with the same methods as Theorem~\ref{thm:inf2-reduce}, along with an additional decoupling argument~\cite[Prop.~1.9]{BT91:Problem-Kadison-Singer}.  We rely on the following corollary in our proof of the Bourgain--Tzafriri theorem.

\begin{cor}
Suppose $\mtx{A}$ is an $n$-column standardized matrix with hollow Gram matrix $\mtx{H} = \mtx{A}^\adj \mtx{A} - \Id$.  Choose $s \leq \lceil \cnst{c} \cdot \strank(\mtx{A}) \rceil$, and write $\delta = s/n$.  Then
$$
\Expect \pnorm{\infty\to1}{\mtx{P}_{\delta} \mtx{H} \mtx{P}_{\delta}}
	\leq \frac{s}{9}.
$$
\end{cor}

\begin{proof}
Suppose $\mtx{A}$ is a standardized matrix with $n$ columns, and define its $n \times n$ hollow Gram matrix $\mtx{H}$.  Observe that the $(\infty, 1)$ norm of $\mtx{H}$ satisfies the bound
$$
\pnorm{\infty\to1}{\mtx{H}}
	\leq n \norm{\mtx{H}}
	\leq n \max\{ \norm{\mtx{A}^\adj \mtx{A}} - 1, 1 \}
	\leq n \normsq{\mtx{A}}.
$$
Meanwhile, the $\pnorm{\rm col}{\cdot}$ norm satisfies
$$
\pnorm{\rm col}{\mtx{H}} < \pnorm{\rm col}{ \mtx{A}^\adj \mtx{A} }
	= \sum\nolimits_j \enorm{ \mtx{A}^\adj \vct{a}_j }
	\leq n \norm{ \mtx{A} }.
$$
These facts play a central role in the calculation.

To continue, invoke the Poissonization result, Proposition~\ref{prop:poisson}, which yields
$$
\Expect \pnorm{\infty\to1}{\mtx{P}_{\delta} \mtx{H} \mtx{P}_{\delta}}
	\leq 2 \pnorm{\infty\to1}{\mtx{R}_{\delta} \mtx{H} \mtx{R}_{\delta}}.
$$
Theorem~\ref{thm:rv-inf1} provides that
$$
\Expect \pnorm{\infty\to1}{\mtx{P}_{\delta} \mtx{H} \mtx{P}_{\delta}}
	\leq \cnst{C} \left[ \delta^2 \pnorm{\infty\to1}{\mtx{H}}
		+ \delta^{3/2} \pnorm{\rm col}{\mtx{H}}  \right]
$$
where we have applied the facts that $\mtx{H}$ is Hermitian and has a zero diagonal.  The two norm bounds result in additional simplifications:
$$
\pnorm{\infty\to1}{\mtx{P}_{\delta} \mtx{H} \mtx{P}_{\delta} }
	\leq \cnst{C} \left[ \delta^2 n \normsq{\mtx{A}}
		+ \delta^{3/2} n \norm{\mtx{A}} \right] \\
	= \cnst{C} s 
	\left[ \delta \normsq{\mtx{A}} + \delta^{1/2} \norm{\mtx{A}} \right].
$$
Since $\mtx{A}$ has unit-norm columns, $\strank(\mtx{A}) = n / \normsq{\mtx{A}}$.  As a result, $\delta = s/n \leq \cnst{c} / \normsq{\mtx{A}}$.  By fixing a sufficiently small constant $\cnst{c}$, we can ensure that
$$
\pnorm{\infty\to1}{\mtx{P}_{\delta} \mtx{H} \mtx{P}_{\delta}} \leq \frac{s}{9},
$$
the advertised bound.
\end{proof}

\section{Entropic Mirror Descent}
	\label{app:emd}
	
The algorithms for the Kashin--Tzafriri theorem and the Bourgain--Tzafriri theorem both require the solution to a convex minimization problem over the probability simplex.
It is important to have a practical algorithm for approaching these optimizations.  To that end, we briefly describe a simple, elegant method called Entropic Mirror Descent \cite{BT03:Mirror-Descent}. We then explain how to apply this technique to the specific objective functions that arise in our work.

\subsection{Convex analysis}

Let $\mathbb{E}$ be a Euclidean space, i.e., a vector space equipped with a real-linear inner product.  Let $\Omega$ be a convex subset of $\mathbb{E}$, and consider a convex function $J : \Omega \to \Rspace{}$.
The \term{subdifferential} $\partial J(\vct{f})$ contains each vector $\vct{\theta} \in \mathbb{E}^\adj$ that satisfies the inequalities
$$
J(\vct{h}) - J(\vct{f})  \geq \ip{ \vct{\theta} }{ \vct{h} - \vct{f} }
	\quad\text{for all $\vct{h} \in \Omega$.}
$$
The elements of the subdifferential are called \term{subgradients}.  They describe the directions and rates of ascent of the function $J$ at the point $\vct{f}$.  When $J$ is differentiable at $\vct{f}$, the gradient is the unique subgradient.


The \term{Lipschitz constant} of the function $J$ with respect to a norm $\triplenorm{\cdot}$ is defined to be the least number $L$ for which
$$
\abs{ J( \vct{h} ) - J( \vct{f} ) }
	\leq L \triplenorm{ \vct{h} - \vct{f} }
\quad\text{for all $\vct{h}, \vct{f} \in \Omega$.}
$$
It can be shown~\cite[Thm.~24.7]{Roc70:Convex-Analysis} that
$$
L = \sup\{ \triplenorm{ \vct{\theta} }_{\adj} :
	\vct{\theta} \in \partial J(\vct{f}), \ \vct{f} \in \Omega \}.
$$
where $\triplenorm{\cdot}_{\adj}$ is the dual norm.


\subsection{Interior subgradient methods}

Consider the (nonsmooth) convex program
$$
\min \ J(\vct{f})
\quad\subjto\quad
\vct{f} \in \Omega.
$$
Subgradient information can be used to solve this problem, but caution is necessary because the negative subgradient is not necessarily a direction of descent.  As a result, subgradient methods are typically \emph{nonmonotone}, which means that the value of the objective function can (and often will) increase.  
It is also common for subgradient methods to produce iterates outside the constraint set.  The classical remedy is to project each iterate back onto the constraint set.  This idea succeeds, but it leads to zigzagging phenomena.

Interior subgradient methods \cite{BT03:Mirror-Descent} are designed to eliminate some of the problematic behavior that projected subgradient methods exhibit.  To develop an interior subgradient method, we need a divergence measure that is tailored to the constraint set.
At each iteration, we perform two steps:
\begin{enumerate}
\item	At the current iterate $\vct{f}$, compute a subgradient $\vct{\theta} \in \partial J(\vct{f})$ to linearize the objective function:
$$
J(\vct{h}) \approx J(\vct{f}) + \ip{\vct{\theta}}{\vct{h} - \vct{f}}
$$

\item	Penalize the linearization with the divergence $D(\cdot; \vct{f})$ from the current iterate, scaled by a (large) parameter $\beta^{-1}$.  Minimize this auxiliary function to produce a new iterate $\vct{f}'$:
$$
\vct{f}' \in \argmin_{\vct{h} \in \Omega} \left\{ J(\vct{f}) + \ip{ \vct{\theta} }{ \vct{h} - \vct{f} } + \beta^{-1} D( \vct{h}; \vct{f} ) \right\}.
$$
\end{enumerate}
The divergence penalty serves two purposes.  First, it ensures that the next iterate is close to the previous iterate, which is essential because the linearization is only useful locally.  Second, it simultaneously prevents the iterates from getting too close to the boundary of the constraint set.  With a careful choice of the parameter $\beta$, we can guarantee progress toward the optimum set, at least on average.

\subsection{Optimization on the probability simplex}

The Entropic Mirror Descent (EMD) algorithm of Beck and Teboulle~\cite{BT03:Mirror-Descent} is a specific instance of the interior subgradient method that is designed for minimizing convex functions over the probability simplex, the set defined by
$$
\Delta_s =
	\{ \vct{f} \in \Rspace{s} : \trace(\vct{f}) = 1, \ \vct{f} \geq \vct{0} \}.
$$
A natural divergence measure for this set is the relative entropy function:
$$
D(\vct{h}; \vct{f}) = \sum\nolimits_{j=1}^s
	h_j \log\left(\frac{h_j}{f_j}\right).
$$
An amazing feature of the resulting interior subgradient method is that the optimization in the second step has a closed form:
$$
h_j = \frac{f_j \exp( - \beta \theta_j )}
	{\sum\nolimits_j f_j \exp(-\beta \theta_j)}
\quad\text{for $j = 1, 2, \dots, s$}.
$$
Algorithm~\ref{alg:emd} describes the procedure that arises from these choices.
Beck and Teboulle have established an elegant efficiency estimate~\cite[Thm.~4.2]{BT03:Mirror-Descent} for this method.

\begin{thm}[Efficiency of EMD] \label{thm:emd}
Let $J : \Delta_s \to \Rspace{}$ be a convex function whose Lipschitz constant with respect to the $\ell_1$ norm is $L$.  The approximate minimizer $\vct{f}$ generated by Algorithm~\ref{alg:emd} satisfies
$$
J( \vct{f} ) - J(\vct{f}_{\star})
	\leq  \sqrt{\frac{2L^2\log s}{T}}
$$
where $\vct{f}_{\star}$ is a minimizer of $J$.
\end{thm}

Algorithm~\ref{alg:emd} succeeds with a wide range of step sizes.  In particular, when the total number $T$ of iterations is unknown, we may compute the step size using the current iteration number $t$:
$$
\beta = \sqrt{ \frac{2\log s}{t \infnorm{\vct{\theta}}^2} }.
$$
This choice increases the right-hand side of the efficiency estimate by a logarithmic factor.

\begin{algorithm}[tb]
\caption{Entropic Mirror Descent} \label{alg:emd}
\centering \fbox{
\begin{minipage}{.9\textwidth} 
\vspace{4pt}
\algname{Emd}{$J$, $s$, $T$}
\alginout{Objective function $J$, dimension $s$, number $T$ of iterations}
{Approximate minimizer $\vct{f}$ of $J$}
\vspace{8pt}\hrule\vspace{8pt}

\begin{algtab}
$\vct{f}^{(1)} = s^{-1} \onevct$ 
	\hfill \{ Initialize with uniform density \} \\ 
\algfor{$t = 1$ {\bf to} $T$}
	Find $\vct{\theta} \in \partial J(\vct{f}^{(t)})$
	\hfill \{ Compute subgradient \} \\
	$\beta = \sqrt{ \frac{2 \log s}{ T \infnorm{\vct{\theta}}^2} }$
	\hfill \{ Compute step size \} \\
	$\vct{h} = \vct{f}^{(t)} \cdot \exp( - \beta \vct{\theta} )$
	\hfill \{ Reweight current iterate \} \\
	$\vct{f}^{(t+1)} = \vct{h} / \trace(\vct{h})$
	\hfill \{ Rescale to obtain next iterate \} \\
\algend {\bf end} {\it for} \\
Return $\vct{f} \in \argmin_t J(\vct{f}^{(t)})$ \\
\end{algtab}
\vspace{6pt}
\end{minipage}}
\end{algorithm}


\subsection{Pietsch factorization via EMD}

Suppose $\mtx{B}$ is a matrix with $s$ columns.  We can rephrase the Pietsch factorization problem \eqref{eqn:pietsch-opt} as an optimization over the probability simplex.  Define the linear operator
$$
\diag : \Rspace{s} \to \Rspace{s \times s}
$$
that maps vectors to diagonal matrices in the obvious way.  We can write the convex program as
\begin{equation} \label{eqn:J-fn}
\min \
\lambda_{\max}( \mtx{B}^\adj\mtx{B} - \alpha^2 \diag(\vct{f}) )
\quad\subjto\quad
\vct{f} \in \Delta_s.
\end{equation}
Abbreviate the objective function $J : \Delta_s \to \Rspace{}$.  We can evidently apply EMD to complete the optimization once we find a way to compute subgradients.

We use methods from the convex analysis of Hermitian matrices to determine the subdifferential of the objective function~\cite{Lew96:Convex-Analysis}.  Let $\mtx{A}$ be an Hermitian matrix.  Then
$$
\partial \lambda_{\max}(\mtx{A}) =
	\conv\{ \vct{u} \vct{u}^\adj :
		\mtx{A}\vct{u} = \lambda_{\max}(\mtx{A}) \vct{u}, \
		\enorm{\vct{u}} = 1 \}.
$$
In words, the subdifferential of the maximum eigenvalue function at $\mtx{A}$ is the convex hull of all rank-one projectors whose range lies in the top eigenspace of $\mtx{A}$.  According to~\cite[Thm.~23.9]{Roc70:Convex-Analysis}, we have
$$
\partial J(\vct{f}) = \{ - \alpha^2 \diag^{\adj}( \mtx{\Theta} ) :
	\mtx{\Theta} \in \partial \lambda_{\max}(\mtx{B}^\adj \mtx{B} - \alpha^2 \diag(\vct{f}) ) \}.
$$
where the adjoint map $\diag^{\adj} : \Rspace{s \times s} \to \Rspace{s}$ extracts the diagonal of a matrix.  In particular, we may construct a subgradient $\vct{\theta} \in \partial J(\vct{f})$ from a normalized maximal eigenvector $\vct{u}$ of the matrix $\mtx{B}^\adj \mtx{B} - \alpha^2 \diag(\vct{f})$ using the formula
$$
\vct{\theta} =
	- \alpha^2 \diag^{\adj}( \vct{u} \vct{u}^\adj ) = - \alpha^2 \abssq{\vct{u}}
$$
where $\abssq{\cdot}$ denotes the componentwise squared magnitude of a vector.

In summary, we can evaluate the objective function $J(\vct{f})$ and simultaneously obtain a subgradient $\vct{\theta} \in \partial J(\vct{f})$ from an eigenvector calculation plus some lower-order operations.  Note that the standard methods for producing a single eigenvector, such as the Lanczos algorithm and its variants \cite[Ch.~9]{GvL96:Matrix-Computations}, require access to the matrix only through its action on vectors.  It is therefore preferable in some settings---for example, when $\mtx{B}$ is sparse---not to form the matrix $\mtx{B}^\adj \mtx{B}$.

Eigenvector computation is a primitive in every numerical linear package, so it is reasonable to assume that high-precision eigenvectors are available.  In any case, slight variants of EMD will work with approximate subgradients, provided they are computed to sufficient precision.  A simple analysis supporting this claim does not seem to be available in the optimization literature, but see \cite[Ch.~6]{Kal07:Efficient-Algorithms} for related work.

We can bound the Lipschitz constant of $J$ with respect to the $\ell_1$ norm just by considering subgradients of the form $\vct{\theta} = - \alpha^2 \abssq{\vct{u}}$ because their convex hull yields the complete subdifferential.  Since the eigenvector $\vct{u}$ is normalized,
$$
\infnorm{ \vct{\theta} } = \alpha^2 \max\nolimits_j \abssq{u_j} \leq \alpha^2,
$$
we determine that the Lipschitz constant $L \leq \alpha^2$.  According to Theorem~\ref{thm:emd}, the EMD algorithm ostensibly requires $\widetilde{\bigO}(\alpha^4)$ iterations to deliver a solution to \eqref{eqn:J-fn} with constant precision.  In practice, far fewer iterations suffice.


\begin{rem}
The application of EMD to \eqref{eqn:J-fn} closely resembles the multiplicative weights method \cite[Ch.~6]{Kal07:Efficient-Algorithms} for solving the {\sc maxcut} problem \eqref{eqn:pietsch-dual}.  Indeed, the two algorithms are substantially identical, except for the specific choice of step sizes and the method for constructing the final solution from the sequence of iterates.  The efficiency estimates are also similar, except that the multiplicative weights method uses the widths of the constraints in lieu of the Lipschitz constant.  EMD appears to be more effective in practice because it exploits the geometry of the problem more completely.
\end{rem}

\subsection{Grothendieck factorization via EMD}

Suppose $\mtx{G}$ is an $s \times s$ Hermitian matrix.  The Grothendieck factorization problem \eqref{eqn:groth-opt} can be expressed as solving
$$
\min \ \lambda_{\max}
	\begin{bmatrix} -\alpha \diag(\vct{f}) & \mtx{G} \\
		\mtx{G} & - \alpha \diag(\vct{f}) \end{bmatrix}
\quad\subjto\quad
\vct{f} \in \Delta_s.
$$
Abbreviate the objective function $J : \Delta_s \to \Rspace{}$.  Once again, EMD is an appropriate technique.

We may obtain subgradients using the same methods as before.  Compute a normalized, maximal eigenvector of the matrix:
$$
\begin{bmatrix} -\alpha \diag(\vct{f}) & \mtx{G} \\
		\mtx{G} & - \alpha \diag(\vct{f}) \end{bmatrix}
\begin{bmatrix} \vct{u} \\ \vct{v} \end{bmatrix}
	= J(\vct{f}) \begin{bmatrix} \vct{u} \\ \vct{v} \end{bmatrix}
\quad\text{where}\quad
\enormsq{\vct{u}} + \enormsq{\vct{v}} = 1.
$$
Then a subgradient $\vct{\theta} \in \partial J(\vct{f})$ can be obtained from the formula
$$
\vct{\theta} = - \alpha \left( \abssq{\vct{u}} + \abssq{\vct{v}} \right).
$$
The Lipschitz constant $L \leq \alpha$, so the number of iterations of EMD is apparently $\widetilde{\bigO}(\alpha^2)$.  Of course, the eigenvector calculations can be streamlined by exploiting the structure of the matrix.

\section*{Acknowledgments}

The author thanks Ben Recht for helpful discussions about eigenvalue minimization.

\bibliographystyle{alpha}
\bibliography{sparse}

\end{document}